\documentclass{amsart}
\usepackage{amsfonts}

\newtheorem{theorem}{Theorem}
\theoremstyle{plain}

\newtheorem{definition}{Definition}

\newtheorem{lemma}{Lemma}

\newtheorem{remark}{Remark}

\numberwithin{equation}{section}
\input{tcilatex}

\begin{document}
\title[The Hermite-Hadamard's inequalities for $\varphi $-convex functions]{%
On Hermite-Hadamard Type Integral Inequalities for Functions Whose Second
Derivative are nonconvex }
\author{Mehmet Zeki SARIKAYA}
\address{Department of Mathematics, \ Faculty of Science and Arts, D\"{u}zce
University, D\"{u}zce-TURKEY}
\email{sarikayamz@gmail.com}
\author{Hakan Bozkurt}
\email{insedi@yahoo.com}
\author{Mehmet Ey\"{u}p KIRIS}
\address{Department of Mathematics, \ Faculty of Science and Arts, Afyon
Kocatepe University, Afyon-TURKEY}
\email{mkiris@gmail.com, kiris@aku.edu.tr}

\begin{abstract}
In this paper, we extend some estimates of the right hand side of a Hermite-
Hadamard type inequality for nonconvex functions whose second derivatives
absolute values are $\varphi $-convex, log-$\varphi $-convex, and quasi-$%
\varphi $-convex.
\end{abstract}

\maketitle

\section{Introduction}

It is well known that if $f$ is a convex function on the interval $I=\left[
a,b\right] $ and $a,b\in I$ with $a<b$, then%
\begin{equation}
f\left( \frac{a+b}{2}\right) \leq \frac{1}{b-a}\int\limits_{a}^{b}f\left(
x\right) dx\leq \frac{f\left( a\right) +f\left( b\right) }{2}  \label{H}
\end{equation}%
which is known as the Hermite-Hadamard inequality for the convex functions.
Both inequalities hold in the reversed direction if $f$ is concave. We note
that Hermite-Hadamard inequality may be regarded as a refinement of the
concept of convexity and it follows easily from Jensen's inequality.
Hermite-Hadamard inequality for convex functions has received renewed
attention in recent years and a remarkable variety of refinements and
generalizations have been found (see, for example, \cite{Alomari}-\cite%
{hassan}, \cite{ion}-\cite{sarikaya4}).

The following lemma was proved for twice differentiable mappings in \cite%
{Dragomir2}:

\begin{lemma}
\label{l1}Let $f:I\subset 
\mathbb{R}
\rightarrow 
\mathbb{R}
$ be a twice differentiable mapping on $I^{o}$, $a,b\in I$ with $a<b$ and $%
f^{\prime \prime }$ of integrable on $[a,b]$, the following equality holds:%
\begin{equation*}
\frac{f(a)+f(b)}{2}+\frac{1}{b-a}\int_{a}^{b}f\left( x\right) dx=\frac{%
\left( b-a\right) ^{2}}{2}\int_{0}^{1}t\left( 1-t\right) f\left( ta+\left(
1-t\right) b\right) dt.
\end{equation*}%
A simple proof of this equality can be also done by twice integrating by
parts in the right hand side.
\end{lemma}

In \cite{hassan}, by using Lemma \ref{l1}, Hussain et al. proved some
inequalities related to Hermite-Hadamard's inequality for $s$-convex
functions:

\begin{theorem}
Let $f:I\subset \lbrack 0,\infty )\rightarrow \mathbb{R}$ be twice
differentiable mapping on $I^{\circ }$ such that $f^{\prime \prime }\in L_{1}%
\left[ a,b\right] ,$ where $a,b\in I$ with $a<b.$ If $\left\vert f^{\prime
\prime }\right\vert $ is $s-$convex on $\left[ a,b\right] $ for some fixed $%
s\in \lbrack 0,1]$ and $q\geq 1,$ then the following inequality holds:%
\begin{equation}
\begin{array}{l}
\left\vert \dfrac{f(a)+f(b)}{2}-\dfrac{1}{b-a}\dint_{a}^{b}f(x)dx\right\vert
\leq \dfrac{\left( b-a\right) ^{2}}{2\times 6^{\frac{1}{p}}}\left[ \dfrac{%
\left\vert f^{\prime \prime }(a)\right\vert ^{q}+\left\vert f^{\prime \prime
}(b)\right\vert ^{q}}{(s+2)(s+3)}\right] ^{\frac{1}{q}},%
\end{array}
\label{H3}
\end{equation}%
where $\frac{1}{p}+\frac{1}{q}=1.$
\end{theorem}

\begin{remark}
If we take $s=1$ in (\ref{H3}), then we have%
\begin{equation*}
\begin{array}{l}
\left\vert \dfrac{f(a)+f(b)}{2}-\dfrac{1}{b-a}\dint_{a}^{b}f(x)dx\right\vert
\leq \dfrac{\left( b-a\right) ^{2}}{12}\left[ \dfrac{\left\vert f^{\prime
\prime }(a)\right\vert ^{q}+\left\vert f^{\prime \prime }(b)\right\vert ^{q}%
}{2}\right] ^{\frac{1}{q}}.%
\end{array}%
\end{equation*}%
We recall that the notion of quasi-convex functions generalizes the notion
of convex functions. More precisely, a function $f:[a,b]\subset \mathbb{R}%
\rightarrow \mathbb{R}$ is said quasi-convex on $[a,b]$ if 
\begin{equation*}
f(tx+(1-t)y)\leq \sup \left\{ f(x),f(y)\right\}
\end{equation*}%
for all $x,y\in \lbrack a,b]$ and $t\in \left[ 0,1\right] .$ Clearly, any
convex function is a quasi-convex function. Furthermore, there exist
quasi-convex functions which are not convex (see \cite{ion}).
\end{remark}

Alomari, Darus and Dragomir in \cite{Alomari} introduced the following
theorems for twice differentiable quasiconvex functions:

\begin{theorem}
Let $f:I\subseteq \mathbb{R}\rightarrow \mathbb{R}$ be a twice
differentiable function on $I^{o}$, $a,b\in I^{o}$ with $a<b$ and $f^{\prime
\prime }$ is integrable on $[a,b]$. If $|f^{\prime \prime }|$ is quasiconvex
on $[a,b]$, then the following inequality holds%
\begin{equation*}
\left\vert \frac{f(a)+f(b)}{2}-\frac{1}{b-a}\int_{a}^{b}f\left( x\right)
dx\right\vert \leq \frac{\left( b-a\right) ^{2}}{12}\max \left\{ \left\vert
f^{\prime \prime }\left( a\right) \right\vert ,\left\vert f^{\prime \prime
}\left( b\right) \right\vert \right\} .
\end{equation*}
\end{theorem}

\begin{theorem}
Let $f:I\subseteq \mathbb{R}\rightarrow \mathbb{R}$ be a twice
differentiable function on $I^{o}$, $a,b\in I^{o}$ with $a<b$ and $f^{\prime
\prime }$ is integrable on $[a,b]$. If $|f^{\prime \prime }|^{\frac{p}{p-1}}$
is a quasiconvex on $[a,b]$, for $p>1$, then the following inequality holds%
\begin{eqnarray*}
&&\left\vert \frac{f(a)+f(b)}{2}-\frac{1}{b-a}\int_{a}^{b}f\left( x\right)
dx\right\vert \\
&& \\
&\leq &\frac{\left( b-a\right) ^{2}}{8}\left( \frac{\sqrt{\pi }}{2}\right) ^{%
\frac{1}{p}}\left( \frac{\Gamma \left( 1+p\right) }{\Gamma \left( \frac{3}{2}%
+p\right) }\right) ^{\frac{1}{p}}\left( \max \left\{ \left\vert f^{\prime
\prime }\left( a\right) \right\vert ^{q},\left\vert f^{\prime \prime }\left(
b\right) \right\vert ^{q}\right\} \right) ^{\frac{1}{q}},
\end{eqnarray*}%
where $\frac{1}{p}+\frac{1}{q}=1.$
\end{theorem}

\begin{theorem}
Let $f:I\subseteq \mathbb{R}\rightarrow \mathbb{R}$ be a twice
differentiable function on $I^{o}$, $a,b\in I^{o}$ with $a<b$ and $f^{\prime
\prime }$ is integrable on $[a,b]$. If $|f^{\prime \prime }|^{q}$ is a
quasiconvex on $[a,b]$, for $q\geq 1$, then the following inequality holds%
\begin{equation*}
\left\vert \frac{f(a)+f(b)}{2}-\frac{1}{b-a}\int_{a}^{b}f\left( x\right)
dx\right\vert \leq \frac{\left( b-a\right) ^{2}}{12}\left( \max \left\{
\left\vert f^{\prime \prime }\left( a\right) \right\vert ^{q},\left\vert
f^{\prime \prime }\left( b\right) \right\vert ^{q}\right\} \right) ^{\frac{1%
}{q}}.
\end{equation*}
\end{theorem}

\section{Preliminaries}

Let $f,\varphi :K\rightarrow 
\mathbb{R}
$, where $K$ is a nonempty closed set in $%
\mathbb{R}
^{n}$, be continuous functions. First of all, we recall the following well
known results and concepts, which are mainly due to Noor and Noor \cite%
{Noor1}\ and Noor \cite{Noor5} as follows:

\begin{definition}
\label{d1} Let $u,v\in K$. Then the set $K$ is said to be $\varphi -convex$
at $u$ with respect to $\varphi $, if%
\begin{equation*}
u+te^{i\varphi }\left( v-u\right) \in K,\text{ }\forall u,v\in K,\text{ }%
t\in \left[ 0,1\right] .
\end{equation*}

\begin{remark}
\label{r1} We would like to mention that Definition \ref{d1} of a $\varphi $%
-convex set has a clear geometric interpretation. This definition
essentially says that there is a path starting from a point $u$ which is
contained in $K$. We do not require that the point $v$ should be one of the
end points of the path. This observation plays an important role in our
analysis. Note that, if we demand that $v$ should be an end point of the
path for every pair of points, $u,v\in K$, then $e^{i\varphi }\left(
v-u\right) =v-u$ if and only if, $\varphi =0$, and consequently $\varphi $%
-convexity reduces to convexity. Thus, it is true that every convex set is
also an $\varphi $-convex set, but the converse is not necessarily true, see 
\cite{Noor1}-\cite{Noor5} and the references therein.
\end{remark}
\end{definition}

\begin{definition}
\label{d2} The function $f$ on the $\varphi $-convex set $K$ is said to be $%
\varphi $-convex with respect to $\varphi $, if%
\begin{equation*}
f\left( u+te^{i\varphi }\left( v-u\right) \right) \leq \left( 1-t\right)
f\left( u\right) +tf\left( v\right) ,\text{ }\forall u,v\in K,\text{ }t\in %
\left[ 0,1\right] .
\end{equation*}%
The function $f$ is said to be $\varphi $-concave if and only if $-f$ is $%
\varphi $-convex. Note that every convex function is a $\varphi $-convex
function, but the converse is not true.
\end{definition}

\begin{definition}
\label{d3} The function $f$ on the $\varphi $-convex set $K$ is said to be
logarithmic $\varphi $-convex with respect to $\varphi $, such that%
\begin{equation*}
f\left( u+te^{i\varphi }\left( v-u\right) \right) \leq \left( f\left(
u\right) \right) ^{1-t}\left( f\left( v\right) \right) ^{t},\text{ }u,v\in K,%
\text{ }t\in \left[ 0,1\right]
\end{equation*}%
where $f\left( .\right) >0$.
\end{definition}

Now we define a new definition for quasi-$\varphi $-convex functions as
follows:

\begin{definition}
\label{d4} The function $f$ on the quasi $\varphi $-convex set $K$ is said
to be quasi $\varphi $-convex with respect to $\varphi $, if%
\begin{equation*}
f\left( u+te^{i\varphi }\left( v-u\right) \right) \leq \max \left\{ f\left(
u\right) ,f\left( v\right) \right\} .
\end{equation*}
\end{definition}

From the above definitions, we have%
\begin{eqnarray*}
f\left( u+te^{i\varphi }\left( v-u\right) \right) &\leq &\left( f\left(
u\right) \right) ^{1-t}\left( f\left( v\right) \right) ^{t} \\
&\leq &\left( 1-t\right) f\left( u\right) +tf\left( v\right) \\
&\leq &\max \left\{ f\left( u\right) ,f\left( v\right) \right\} .
\end{eqnarray*}%
Clearly, any $\varphi $-convex function is a quasi $\varphi $-convex
function. Furthermore, there exist quasi $\varphi $-convex functions which
are neither $\varphi $-convex nor continuous. For example, for%
\begin{equation*}
\varphi (v,u)=\left\{ 
\begin{array}{c}
2k\pi ,\text{ \ \ \ \ \ }u.v\geq 0,\ k\in \mathbb{Z} \\ 
k\pi ,\text{ \ \ \ \ \ \ }u.v<0,\ k\in \mathbb{Z}%
\end{array}%
\right.
\end{equation*}
the floor function $f_{loor}(x)=\left\lfloor x\right\rfloor ,$ is the
largest integer not greater than $x$, is an example of a monotonic
increasing function which is quasi $\varphi $-convex but it is neither $%
\varphi $-convex nor continuous.

In \cite{Noor3},\ Noor proved the Hermite-Hadamard inequality for the $%
\varphi -$convex functions as follows:

\begin{theorem}
\label{tt1} Let $f:K=\left[ a,a+e^{i\varphi }\left( b-a\right) \right]
\rightarrow \left( 0,\infty \right) $ be a $\varphi $-convex\ function on
the interval of real numbers $K^{0}$ (the interior of $K$) and $a,b\in K^{0}$
with $a<a+e^{i\varphi }\left( b-a\right) $ and $0\leq \varphi \leq \frac{\pi 
}{2}$. Then the following inequality holds:%
\begin{eqnarray}
f\left( \frac{2a+e^{i\varphi }\left( b-a\right) }{2}\right) &\leq &\frac{1}{%
e^{i\varphi }\left( b-a\right) }\dint\limits_{a}^{a+e^{i\varphi }\left(
b-a\right) }f\left( x\right) dx  \label{2} \\
&\leq &\frac{f\left( a\right) +f\left( a+e^{i\varphi }\left( b-a\right)
\right) }{2}\leq \frac{f\left( a\right) +f\left( b\right) }{2}.  \notag
\end{eqnarray}
\end{theorem}

This inequality can easily show that using the $\varphi $-convex\ function's
definition and $f\left( a+e^{i\varphi }\left( b-a\right) \right) <f\left(
b\right) .$

In \cite{sarikaya6} and \cite{sarikaya7}, the authors proved some
generalization inequalities connected with Hermite-Hadamard's inequality for
diferentiable $\varphi $-convex functions.

In this article, using functions whose second derivatives absolute values
are $\varphi $-convex, log-$\varphi $-convex and quasi-$\varphi $-convex, we
obtained new inequalities related to the right side of Hermite-Hadamard
inequality given with (\ref{2}).

\section{Hermite-Hadamard Type Inequalities}

We will start the following theorem:

\begin{theorem}
Let $K\subset 
\mathbb{R}
$ be an open interval, $a,a+e^{i\varphi }(b-a)\in K$ with $a<b$ and $f:K=%
\left[ a,a+e^{i\varphi }(b-a)\right] \rightarrow (0,\infty )$ a twice
differentiable mapping such that $f^{\prime \prime }$ is integrable and $%
0\leq \varphi \leq \frac{\pi }{2}$. If $\left\vert f^{\prime \prime
}\right\vert $ is $\varphi $-convex function on $\left[ a,a+e^{i\varphi
}(b-a)\right] $. Then, the following inequality holds:%
\begin{eqnarray*}
&&\left\vert \frac{1}{e^{i\varphi }(b-a)}\int_{a}^{a+e^{i\varphi
}(b-a)}f(x)dx-\frac{f(a)+f(a+e^{i\varphi }(b-a))}{2}\right\vert \\
&\leq &\frac{e^{2i\varphi }(b-a)^{2}}{24}\left[ \left\vert f^{\prime \prime
}(a)\right\vert +\left\vert f^{\prime \prime }(b)\right\vert \right] .
\end{eqnarray*}
\end{theorem}

\begin{proof}
If the partial integration method is applied twice, then it follows that%
\begin{eqnarray}
&&\frac{e^{2i\varphi }(b-a)^{2}}{2}\int_{0}^{1}(t-t^{2})f^{\prime \prime
}(a+te^{i\varphi }(b-a))dt  \label{11} \\
&&  \notag \\
&=&\frac{1}{e^{i\varphi }(b-a)}\int_{a}^{a+e^{i\varphi }(b-a)}f(x)dx-\frac{%
f(a)+f(a+e^{i\varphi }(b-a))}{2}.  \notag
\end{eqnarray}%
Thus, by $\varphi $-convexity\ function of \ $\left\vert f^{\prime \prime
}\right\vert $, we have%
\begin{eqnarray*}
&&\left\vert \frac{1}{e^{i\varphi }(b-a)}\int_{a}^{a+e^{i\varphi
}(b-a)}f(x)dx-\frac{f(a)+f(a+e^{i\varphi }(b-a))}{2}\right\vert \\
&\leq &\frac{e^{2i\varphi }(b-a)^{2}}{2}\left\vert
\int_{0}^{1}(t-t^{2})f^{\prime \prime }(a+te^{i\varphi }(b-a))dt\right\vert
\\
&\leq &\frac{e^{2i\varphi }(b-a)^{2}}{2}\int_{0}^{1}(t-t^{2})\left[
(1-t)\left\vert f^{\prime \prime }(a)\right\vert +t\left\vert f^{\prime
\prime }(b)\right\vert \right] dt \\
&\leq &\frac{e^{2i\varphi }(b-a)^{2}}{24}\left[ \left\vert f^{\prime \prime
}(a)\right\vert +\left\vert f^{\prime \prime }(b)\right\vert \right]
\end{eqnarray*}

which the proof is completed.
\end{proof}

\begin{theorem}
Let $f:K=\left[ a,a+e^{i\varphi }(b-a)\right] \rightarrow (0,\infty )$ be a
twice differentiable mapping on $K^{0}$ and $f^{\prime \prime }$ be
integrable on $\left[ a,a+e^{i\varphi }(b-a)\right] $. Assume $p\in 
\mathbb{R}
$ with $p>1$. If $\left\vert f^{\prime \prime }\right\vert ^{p/p-1}$ is $%
\varphi $-convex function on the interval of real numbers $K^{0}$ (the
interior of $K$) and $a,b\in K^{0}$ with $a<$ $a+e^{i\varphi }(b-a)$ and $%
0\leq \varphi \leq \frac{\pi }{2}$. Then, the following inequality holds:%
\begin{eqnarray*}
&&\left\vert \frac{1}{e^{i\varphi }(b-a)}\int_{a}^{a+e^{i\varphi
}(b-a)}f(x)dx-\frac{f(a)+f(a+e^{i\varphi }(b-a))}{2}\right\vert \\
&& \\
&\leq &\frac{e^{2i\varphi }(b-a)^{2}}{8}\left( \frac{\sqrt{\pi }}{2}\right)
^{\frac{1}{p}}\left( \frac{\Gamma (p+1)}{\Gamma (\frac{3}{2}+p)}\right) ^{%
\frac{1}{p}}\left( \frac{\left\vert f^{\prime \prime }(a)\right\vert ^{\frac{%
p}{p-1}}+\left\vert f^{\prime \prime }(b)\right\vert ^{\frac{p}{p-1}}}{2}%
\right) ^{\frac{p-1}{p}}.
\end{eqnarray*}
\end{theorem}

\begin{proof}
By assumption, H\"{o}lder's inequality and (\ref{11}), we have%
\begin{eqnarray*}
&&\left\vert \frac{1}{e^{i\varphi }(b-a)}\int_{a}^{a+e^{i\varphi
}(b-a)}f(x)dx-\frac{f(a)+f(a+e^{i\varphi }(b-a))}{2}\right\vert \\
&\leq &\frac{e^{2i\varphi }(b-a)^{2}}{2}\int_{0}^{1}\left\vert
t-t^{2}\right\vert \left\vert f^{\prime \prime }(a+te^{i\varphi
}(b-a))\right\vert dt \\
&\leq &\frac{e^{2i\varphi }(b-a)^{2}}{2}\left(
\int_{0}^{1}(t-t^{2})^{p}dt\right) ^{\frac{1}{p}}\left(
\int_{0}^{1}\left\vert f^{\prime \prime }(a+te^{i\varphi }(b-a))\right\vert
^{\frac{p}{p-1}}dt\right) ^{\frac{p-1}{p}} \\
&\leq &\frac{e^{2i\varphi }(b-a)^{2}}{2}\left( \frac{2^{-1-2p}\sqrt{\pi }%
\Gamma (p+1)}{\Gamma (\frac{3}{2}+p)}\right) ^{\frac{1}{p}}\left(
\int_{0}^{1}\left[ (1-t)\left\vert f^{\prime \prime }(a)\right\vert ^{\frac{p%
}{p-1}}+t\left\vert f^{\prime \prime }(b)\right\vert ^{\frac{p}{p-1}}\right]
dt\right) ^{\frac{p-1}{p}} \\
&=&\frac{e^{2i\varphi }(b-a)^{2}}{8}\left( \frac{\sqrt{\pi }}{2}\right) ^{%
\frac{1}{p}}\left( \frac{\Gamma (p+1)}{\Gamma (\frac{3}{2}+p)}\right) ^{%
\frac{1}{p}}\left( \frac{\left\vert f^{\prime \prime }(a)\right\vert ^{\frac{%
p}{p-1}}+\left\vert f^{\prime \prime }(b)\right\vert ^{\frac{p}{p-1}}}{2}%
\right) ^{\frac{p-1}{p}}
\end{eqnarray*}

where we use the fact that%
\begin{equation*}
\int_{0}^{1}(t-t^{2})^{p}dt=\frac{2^{-1-2p}\sqrt{\pi }\Gamma (p+1)}{\Gamma (%
\frac{3}{2}+p)}
\end{equation*}

which completes the proof.
\end{proof}

Let us denote by $A(a,b)$ the arithmetic mean of the nonnegative real
numbers, and by $L(a,b)$ the logaritmic mean of the same numbers.

\begin{theorem}
\label{as} Let $K\subset 
\mathbb{R}
$ be an open interval, $a,a+e^{i\varphi }(b-a)\in K$ with $a<b$ and $f:K=%
\left[ a,a+e^{i\varphi }(b-a)\right] \rightarrow (0,\infty )$ a twice
differentiable mapping such that $f^{\prime \prime }$ is integrable and $%
0\leq \varphi \leq \frac{\pi }{2}$. If $\left\vert f^{\prime \prime
}\right\vert $ is log $\varphi $-convex function on $\left[ a,a+e^{i\varphi
}(b-a)\right] $. Then, the following inequality holds:%
\begin{eqnarray*}
&&\left\vert \frac{1}{e^{i\varphi }(b-a)}\int_{a}^{a+e^{i\varphi
}(b-a)}f(x)dx-\frac{f(a)+f(a+e^{i\varphi }(b-a))}{2}\right\vert \\
&\leq &\left( \frac{e^{i\varphi }(b-a)}{\log \left\vert f^{\prime \prime
}(b)\right\vert -\log \left\vert f^{\prime \prime }(a)\right\vert }\right)
^{2}\left[ A\left( \left\vert f^{\prime \prime }(b)\right\vert ,\left\vert
f^{\prime \prime }(a)\right\vert \right) -L\left( \left\vert f^{\prime
\prime }(b)\right\vert ,\left\vert f^{\prime \prime }(a)\right\vert \right) %
\right] .
\end{eqnarray*}
\end{theorem}

\begin{proof}
By using (\ref{11}) and log $\varphi $-convexity\ of \ $\left\vert f^{\prime
\prime }\right\vert $, we have 
\begin{eqnarray*}
&&\left\vert \frac{1}{e^{i\varphi }(b-a)}\int_{a}^{a+e^{i\varphi
}(b-a)}f(x)dx-\frac{f(a)+f(a+te^{i\varphi }(b-a))}{2}\right\vert \\
&\leq &\frac{e^{2i\varphi }(b-a)^{2}}{2}\int_{0}^{1}(t-t^{2})\left\vert
f^{\prime \prime }(a+te^{i\varphi }(b-a))\right\vert dt \\
&\leq &\frac{e^{2i\varphi }(b-a)^{2}}{2}\int_{0}^{1}(t-t^{2})\left(
\left\vert f^{\prime \prime }(a)\right\vert ^{1-t}\left\vert f^{\prime
\prime }(b)\right\vert ^{t}\right) dt \\
&=&\frac{e^{2i\varphi }(b-a)^{2}}{2}\left[ \frac{\left\vert f^{\prime \prime
}(b)\right\vert +\left\vert f^{\prime \prime }(a)\right\vert }{\left( \log
\left\vert f^{\prime \prime }(b)\right\vert -\log \left\vert f^{\prime
\prime }(a)\right\vert \right) ^{2}}-\frac{2\left( \left\vert f^{\prime
\prime }(b)\right\vert -\left\vert f^{\prime \prime }(a)\right\vert \right) 
}{\left( \log \left\vert f^{\prime \prime }(b)\right\vert -\log \left\vert
f^{\prime \prime }(a)\right\vert \right) ^{3}}\right] \\
&=&\left( \frac{e^{i\varphi }(b-a)}{\log \left\vert f^{\prime \prime
}(b)\right\vert -\log \left\vert f^{\prime \prime }(a)\right\vert }\right)
^{2}\left[ A\left( \left\vert f^{\prime \prime }(b)\right\vert ,\left\vert
f^{\prime \prime }(a)\right\vert \right) -L\left( \left\vert f^{\prime
\prime }(b)\right\vert ,\left\vert f^{\prime \prime }(a)\right\vert \right) %
\right] .
\end{eqnarray*}%
The proof of Theorem \ref{as} is completed.
\end{proof}

\begin{theorem}
Let $f:K=\left[ a,a+e^{i\varphi }(b-a)\right] \rightarrow (0,\infty )$ be a
twice differentiable mapping on $K^{o}$ and $f^{\prime \prime }$ be
integrable on $\left[ a,a+e^{i\varphi }(b-a)\right] $. Assume $p\in 
\mathbb{R}
$ with $p>1$. If $\left\vert f^{\prime \prime }\right\vert ^{p/p-1}$ is log $%
\varphi $-convex function on the interval of real numbers $K^{o}$ (the
interior of $K$) and $a,b\in K^{o}$ with $a<$ $a+e^{i\varphi }(b-a)$ and $%
0\leq \varphi \leq \frac{\pi }{2}$. Then, the following inequality holds:%
\begin{eqnarray*}
&&\left\vert \frac{1}{e^{i\varphi }(b-a)}\int_{a}^{a+e^{i\varphi
}(b-a)}f(x)dx-\frac{f(a)+f(a+e^{i\varphi }(b-a))}{2}\right\vert \\
&& \\
&\leq &\frac{e^{2i\varphi }(b-a)^{2}}{8}\left( \frac{\sqrt{\pi }}{2}\right)
^{\frac{1}{p}}\left( \frac{\Gamma (p+1)}{\Gamma (\frac{3}{2}+p)}\right) ^{%
\frac{1}{p}}\left( \frac{p-1}{p}\right) ^{\frac{p-1}{p}}\left( \frac{%
\left\vert f^{\prime \prime }(a)\right\vert ^{\frac{p}{p-1}}-\left\vert
f^{\prime \prime }(b)\right\vert ^{\frac{p}{p-1}}}{\log \left\vert f^{\prime
\prime }(b)\right\vert -\log \left\vert f^{\prime \prime }(a)\right\vert }%
\right) ^{\frac{p-1}{p}}.
\end{eqnarray*}
\end{theorem}

\begin{proof}
By using (\ref{11}) and the well known H\"{o}lder's integral inequality, we
obtain%
\begin{eqnarray*}
&&\left\vert \frac{1}{e^{i\varphi }(b-a)}\int_{a}^{a+e^{i\varphi
}(b-a)}f(x)dx-\frac{f(a)+f(a+e^{i\varphi }(b-a))}{2}\right\vert \\
&\leq &\frac{e^{2i\varphi }(b-a)^{2}}{2}\int_{0}^{1}(t-t^{2})\left\vert
f^{\prime \prime }(a+te^{i\varphi }(b-a))\right\vert dt \\
&\leq &\frac{e^{2i\varphi }(b-a)^{2}}{2}\left(
\int_{0}^{1}(t-t^{2})^{p}dt\right) ^{\frac{1}{p}}\left(
\int_{0}^{1}\left\vert f^{\prime \prime }(a+te^{i\varphi }(b-a))\right\vert
^{\frac{p}{p-1}}dt\right) ^{\frac{p-1}{p}} \\
&\leq &\frac{e^{2i\varphi }(b-a)^{2}}{2}\left( \frac{2^{-1-2p}\sqrt{\pi }%
\Gamma (p+1)}{\Gamma (\frac{3}{2}+p)}\right) ^{\frac{1}{p}}\left(
\int_{0}^{1}\left\vert f^{\prime \prime }(a)\right\vert ^{\frac{p}{p-1}%
(1-t)}\left\vert f^{\prime \prime }(b)\right\vert ^{\frac{p}{p-1}t}dt\right)
^{\frac{p-1}{p}} \\
&=&\frac{e^{2i\varphi }(b-a)^{2}}{8}\left( \frac{\sqrt{\pi }}{2}\right) ^{%
\frac{1}{p}}\left( \frac{\Gamma (p+1)}{\Gamma (\frac{3}{2}+p)}\right) ^{%
\frac{1}{p}}\left( \frac{p-1}{p}\right) ^{\frac{p-1}{p}}\left( \frac{%
\left\vert f^{\prime \prime }(a)\right\vert ^{\frac{p}{p-1}}-\left\vert
f^{\prime \prime }(b)\right\vert ^{\frac{p}{p-1}}}{\log \left\vert f^{\prime
\prime }(b)\right\vert -\log \left\vert f^{\prime \prime }(a)\right\vert }%
\right) ^{\frac{p-1}{p}}.
\end{eqnarray*}
\end{proof}

\begin{theorem}
Let $f:K=\left[ a,a+e^{i\varphi }(b-a)\right] \rightarrow (0,\infty )$ be a
differentiable mapping on $K^{0}$ and $f^{\prime \prime }$ be integrable on $%
\left[ a,a+e^{i\varphi }(b-a)\right] $. If $\left\vert f^{\prime \prime
}\right\vert $ is a quasi $\varphi $-convex function on the interval of real
numbers $K^{o}$ (the interior of $K$) and $a,b\in K^{o}$ with $%
a<a+e^{i\varphi }(b-a)$ and $0\leq \varphi \leq \frac{\pi }{2}$. Then, the
following inequality holds:%
\begin{eqnarray*}
&&\left\vert \frac{1}{e^{i\varphi }(b-a)}\int_{a}^{a+e^{i\varphi
}(b-a)}f(x)dx-\frac{f(a)+f(a+te^{i\varphi }(b-a))}{2}\right\vert \\
&& \\
&\leq &\frac{e^{i\varphi }(b-a)}{4}\max \{\left\vert f^{\prime
}(a)\right\vert ,\left\vert f^{\prime }(b)\right\vert \}.
\end{eqnarray*}
\end{theorem}

\begin{proof}
By using (\ref{11}) and the quasi $\varphi $-convexity of \ $\left\vert
f^{\prime \prime }\right\vert $, we have%
\begin{eqnarray*}
&&\left\vert \frac{1}{e^{i\varphi }(b-a)}\int_{a}^{a+e^{i\varphi
}(b-a)}f(x)dx-\frac{f(a)+f(a+te^{i\varphi }(b-a))}{2}\right\vert \\
&\leq &\frac{e^{2i\varphi }(b-a)^{2}}{2}\int_{0}^{1}(t-t^{2})\left\vert
f^{\prime \prime }(a+te^{i\varphi }(b-a))\right\vert dt \\
&\leq &\frac{e^{2i\varphi }(b-a)^{2}}{2}\max \{\left\vert f^{\prime \prime
}(a)\right\vert ,\left\vert f^{\prime \prime }(b)\right\vert
\}\int_{0}^{1}(t-t^{2})dt \\
&\leq &\frac{e^{2i\varphi }(b-a)^{2}}{24}\max \{\left\vert f^{\prime \prime
}(a)\right\vert ,\left\vert f^{\prime \prime }(b)\right\vert \}.
\end{eqnarray*}
\end{proof}

\begin{theorem}
Let $f:K=\left[ a,a+e^{i\varphi }(b-a)\right] \rightarrow (0,\infty )$ be a
differentiable mapping on $K^{o}$and $f^{\prime \prime }$ be integrable on $%
\left[ a,a+e^{i\varphi }(b-a)\right] $. Assume $p\in 
\mathbb{R}
$ with $p>1$. If $\left\vert f^{\prime \prime }\right\vert ^{p/p-1}$ is a
quasi $\varphi $-convex function on the interval of real numbers $K^{o}$
(the interior of $K$) and $a,b\in K^{o}$ with $a<$ $a+e^{i\varphi }(b-a)$
and $0\leq \varphi \leq \frac{\pi }{2}$. Then, the following inequality
holds:%
\begin{eqnarray*}
&&\left\vert \frac{1}{e^{i\varphi }(b-a)}\int_{a}^{a+e^{i\varphi
}(b-a)}f(x)dx-\frac{f(a)+f(a+te^{i\varphi }(b-a))}{2}\right\vert \\
&\leq &\frac{e^{2i\varphi }(b-a)^{2}}{8}\left( \frac{\sqrt{\pi }}{2}\right)
^{\frac{1}{p}}\left( \frac{\Gamma (p+1)}{\Gamma (\frac{3}{2}+p)}\right) ^{%
\frac{1}{p}}\left[ \max \{\left\vert f^{\prime \prime }(a)\right\vert ^{%
\frac{p}{p-1}},\left\vert f^{\prime \prime }(b)\right\vert ^{\frac{p}{p-1}}\}%
\right] ^{\frac{p-1}{p}}.
\end{eqnarray*}
\end{theorem}

\begin{proof}
By using (\ref{11}) and the well known H\"{o}lder's integral inequality, we
get%
\begin{eqnarray*}
&&\left\vert \frac{1}{e^{i\varphi }(b-a)}\int_{a}^{a+e^{i\varphi
}(b-a)}f(x)dx-\frac{f(a)+f(a+te^{i\varphi }(b-a))}{2}\right\vert \\
&\leq &\frac{e^{2i\varphi }(b-a)^{2}}{2}\int_{0}^{1}(t-t^{2})\left\vert
f^{\prime \prime }(a+te^{i\varphi }(b-a))\right\vert dt \\
&\leq &\frac{e^{i\varphi }(b-a)}{2}\left( \int_{0}^{1}(t-t^{2})^{p}dt\right)
^{\frac{1}{p}}\left( \int_{0}^{1}\left\vert f^{\prime }(a+te^{i\varphi
}(b-a))\right\vert ^{\frac{p-1}{p}}dt\right) ^{\frac{p}{p-1}} \\
&\leq &\frac{e^{i\varphi }(b-a)}{2}\left( \frac{2^{-1-2p}\sqrt{\pi }\Gamma
(p+1)}{\Gamma (\frac{3}{2}+p)}\right) ^{\frac{1}{p}}\left( \int_{0}^{1}\max
\{\left\vert f^{\prime }(a)\right\vert ^{\frac{p}{p-1}},\left\vert f^{\prime
}(b)\right\vert ^{\frac{p}{p-1}}\}dt\right) ^{\frac{p}{p-1}} \\
&\leq &\frac{e^{2i\varphi }(b-a)^{2}}{8}\left( \frac{\sqrt{\pi }}{2}\right)
^{\frac{1}{p}}\left( \frac{\Gamma (p+1)}{\Gamma (\frac{3}{2}+p)}\right) ^{%
\frac{1}{p}}\left[ \max \{\left\vert f^{\prime \prime }(a)\right\vert ^{%
\frac{p}{p-1}},\left\vert f^{\prime \prime }(b)\right\vert ^{\frac{p}{p-1}}\}%
\right] ^{\frac{p-1}{p}}.
\end{eqnarray*}
\end{proof}

\begin{theorem}
Let $f:K=\left[ a,a+e^{i\varphi }(b-a)\right] \rightarrow (0,\infty )$ be a
differentiable mapping on $K^{o}$and $f^{\prime \prime }$ be integrable on $%
\left[ a,a+e^{i\varphi }(b-a)\right] $. Assume $q\in 
\mathbb{R}
$ with $q\geq 1$. If $\left\vert f^{\prime \prime }\right\vert ^{q}$ is a
quasi $\varphi $-convex function on the interval of real numbers $K^{o}$
(the interior of $K$) and $a,b\in K^{o}$ with $a<$ $a+e^{i\varphi }(b-a)$
and $0\leq \varphi \leq \frac{\pi }{2}$. Then, the following inequality
holds:%
\begin{eqnarray*}
&&\left\vert \frac{1}{e^{i\varphi }(b-a)}\int_{a}^{a+e^{i\varphi
}(b-a)}f(x)dx-\frac{f(a)+f(a+te^{i\varphi }(b-a))}{2}\right\vert \\
&\leq &\frac{e^{2i\varphi }(b-a)^{2}}{12}\left[ \max \{\left\vert f^{\prime
\prime }(a)\right\vert ^{q},\left\vert f^{\prime \prime }(b)\right\vert
^{q}\}\right] ^{\frac{1}{q}}.
\end{eqnarray*}
\end{theorem}

\begin{proof}
By using (\ref{l1}) and the well known power mean integral inequality, we
have%
\begin{eqnarray*}
&&\left\vert \frac{1}{e^{i\varphi }(b-a)}\int_{a}^{a+e^{i\varphi
}(b-a)}f(x)dx-\frac{f(a)+f(a+te^{i\varphi }(b-a))}{2}\right\vert \\
&\leq &\frac{e^{2i\varphi }(b-a)^{2}}{2}\int_{0}^{1}(t-t^{2})\left\vert
f^{\prime \prime }(a+te^{i\varphi }(b-a))\right\vert dt \\
&\leq &\frac{e^{i\varphi }(b-a)}{2}\left( \int_{0}^{1}(t-t^{2})dt\right) ^{%
\frac{1}{p}}\left( \int_{0}^{1}(t-t^{2})\left\vert f^{\prime
}(a+te^{i\varphi }(b-a))\right\vert ^{q}dt\right) ^{\frac{1}{q}} \\
&\leq &\frac{e^{i\varphi }(b-a)}{2}\left( \frac{1}{6}\right) ^{\frac{1}{p}%
}\left( \max \{\left\vert f^{\prime }(a)\right\vert ^{q},\left\vert
f^{\prime }(b)\right\vert ^{q}\}\int_{0}^{1}(t-t^{2})dt\right) ^{\frac{1}{q}}
\\
&\leq &\frac{e^{2i\varphi }(b-a)^{2}}{12}\left[ \max \{\left\vert f^{\prime
\prime }(a)\right\vert ^{q},\left\vert f^{\prime \prime }(b)\right\vert
^{q}\}\right] ^{\frac{1}{q}},
\end{eqnarray*}%
where $\frac{1}{p}+\frac{1}{q}=1.$
\end{proof}

\end{document}